\documentclass[a4paper]{article}

\usepackage[utf8]{inputenc}
\usepackage{lmodern}
\usepackage{amssymb,amsfonts,amsthm,amsmath,bbm,mathrsfs,aliascnt,mathtools,pifont}
\usepackage[english]{babel}
\usepackage[colorlinks]{hyperref}
\usepackage{tikz}
\usetikzlibrary{decorations}
\usetikzlibrary{decorations.pathreplacing}
\tikzset{individu/.style={draw,thick}}

\newcounter{dummy} \numberwithin{dummy}{section}
\newtheorem{theorem}[dummy]{Theorem}
\newtheorem{lemma}[dummy]{Lemma}
\newtheorem{proposition}[dummy]{Proposition}

\newtheorem{remark}[dummy]{Remark}

\usepackage{bm} 
\usepackage{amsmath}
\usepackage{graphicx}

\usepackage{xcolor}
\usepackage{hyperref} 
\hypersetup{backref=true,       
    pagebackref=true,               
    hyperindex=true,                
    colorlinks=true,                
    breaklinks=true,                
    urlcolor= black,                
    linkcolor= blue,                
    bookmarks=true,                 
    bookmarksopen=false,
    filecolor=black,
    citecolor=green,
    linkbordercolor=blue
}
\def\MR#1{\href{http://www.ams.org/mathscinet-getitem?mr=#1}{MR#1}}

\usepackage{amsthm}
\usepackage{amsmath}
\usepackage{amssymb}
\usepackage{mathrsfs}
\usepackage{graphicx}
\usepackage{mwe}
\usepackage{subfig}
\usepackage{caption}
\usepackage{subcaption}
\usepackage{ucs}
\usepackage{yfonts}
\usepackage{bbm}
\usepackage{graphicx}
\usepackage{enumerate}
\usepackage{xifthen}
\usepackage{upgreek} 
\usepackage{stmaryrd} 
\usepackage{xcolor}
\usepackage{mathrsfs}
\usepackage{bm} 
\usepackage{mathabx} 
\usepackage{geometry}
\usepackage{indentfirst}
\usepackage{enumitem}

\usepackage{tcolorbox} 

\numberwithin{equation}{section}

\title{ How often does a critical elephant random walk\\ return to origin}
\author{Zheng Fang\thanks{zheng.fang@math.uzh.ch} }
\date{Institute of Mathematics, University of Zurich\thanks{Institute of mathematics, Winterthurerstrasse 190, 8057 Zürich, Switzerland.}}
\begin{document}

\maketitle

\begin{abstract}
    We study the asymptotic behaviour of the number of times the elephant random walk in the critical regime visits the origin. Our result entails that most zeros of the critical elephant random walk occur shortly before its last passage time at zero. Additionally, we derive the tail estimate of the first return time of the random walk.
\end{abstract}

\section{Introduction} \label{section:introduction}

 The $1$-dimensional elephant random walk, for which we use the abbreviation ERW, was introduced by Sch\"utz and Trimper in \cite{elephantsalwaysremember}. This process, denoted here by $(S(n))_{n\geq 0}$, is an $\mathbb{Z}$-valued nearest neighbour process started from $S(0)=0$ a.s and depends on a memory parameter $p\in [0, 1]$ . First, $S(1)$ takes value $1$ with some probability $q\in[0,1]$, and value $-1$ with probability $1-q$. From the second step onwards, at each step $n$, we choose a number $u(n)$ uniformly at random from previous times $\{1, 2, \cdots, n-1\}$. With probability $p$, we repeat the selected step, or with probability $1-p$, we repeat the opposite of the chosen step. It is well-known that ERW demonstrates a phase transition at $p=3/4$. When $p<3/4$, the ERW is in diffusive regime, i.e. the mean square displacement of subcritical ERW grows linearly with time, while $p\geq 3/4$, it exhibits superdiffusive behaviour, meaning the mean square displacement of critical and supercritical ERW grow faster than linearly with time. 
Its asymptotic behaviour in different regimes was extensively studied, such as its law of large numbers, central limit theorems and invariance principles, for example, see \cite{{connectiontourns}, {martingale approach}, {CLT}, {invariance}}. The recurrence and transience regimes of $1$-dimensional ERW was established in \cite{recurrence for 1d ERW}, and we know for both diffusive and marginally superdiffusive regime (i.e. $p\leq 3/4$), the ERW is recurrent and its scaling limit converge in law to a Gaussian random variable. In specific, in the diffusive regime, we have $\left(n^{-1/2}S(\lfloor{nt}\rfloor)\right)_{t\geq 0}$ converges in distribution to a centered Gaussian process as $n\rightarrow\infty$, and in the critical regime, we have the following weak convergence in Skorokhod topology on $\mathbb{R}_+$ as $n\rightarrow\infty$,
\begin{equation}\label{intro: critical regime convergence}
    \left(\frac{S\left({\lfloor{n^t}\rfloor}\right)}{\sqrt{\log n}n^{t/2}}, t\geq 0\right)\xrightarrow[n\rightarrow\infty]{(d)} (B(t), t\geq 0)
\end{equation}
 where $B$ is a standard Brownian motion. However, the ERW is transient in the supercritical regime (i.e. for $p>3/4$). And in this case we have $\left(S\left({\lfloor{nt}\rfloor}\right)/n^{2p-1}\right)_{t\geq 0}$ converges to $\left(t^{2p-1}L\right)_{t\geq 0}$ almost surely as $n\rightarrow\infty$, where $L$ is a non-degenerate random variable with mean zero, and is non-Gaussian.

We define the number of times the ERW visits origin up to the $n$-th step, namely
\begin{equation*}
    Z(n):=\text{Card}\{1\leq k\leq n: S(k)=0\}
\end{equation*}

The asymptotic behaviour of $Z(n)$ has been studied in the diffusive regime by Bertoin in \cite{countingzeros}, in which he shows
\begin{equation}\label{intro: diffusive regime}
    \lim_{n\rightarrow\infty}\frac{Z(n)}{\sqrt{n}}=V \quad\text{in distribution}
\end{equation}
where $V$ is a random variable that is non-zero almost surely. Additionally, here we can verify the ERW is recurrent at origin for $p\leq 3/4$ from (\ref{intro: critical regime convergence}) and (\ref{intro: diffusive regime}). By noting that the ERW is transient in the supercritical regime, we obtain $Z(n)$ is finite a.s. The purpose of this work is to treat the critical case for the counting zero process $Z(n)$. Specifically, we show that
\begin{theorem}\label{intro: theorem 1}
    \begin{equation}\label{intro:convergence of Z(n)}
    \lim_{n\rightarrow\infty} \log Z(n)/ \log n = A/2 \quad \text{in distribution}
\end{equation}
where $A$ has arcsine distribution, i.e.
\begin{equation*}
    \lim_{n\rightarrow\infty}\mathbb{P}\left(Z(n)\leq n^{a}\right)=\frac{2}{\pi}\arcsin\left(\sqrt{2a}\right) \qquad\text{for  $\forall a\in [0,1/2]$}
\end{equation*}
\end{theorem}

\begin{figure}
  \includegraphics[width=\linewidth]{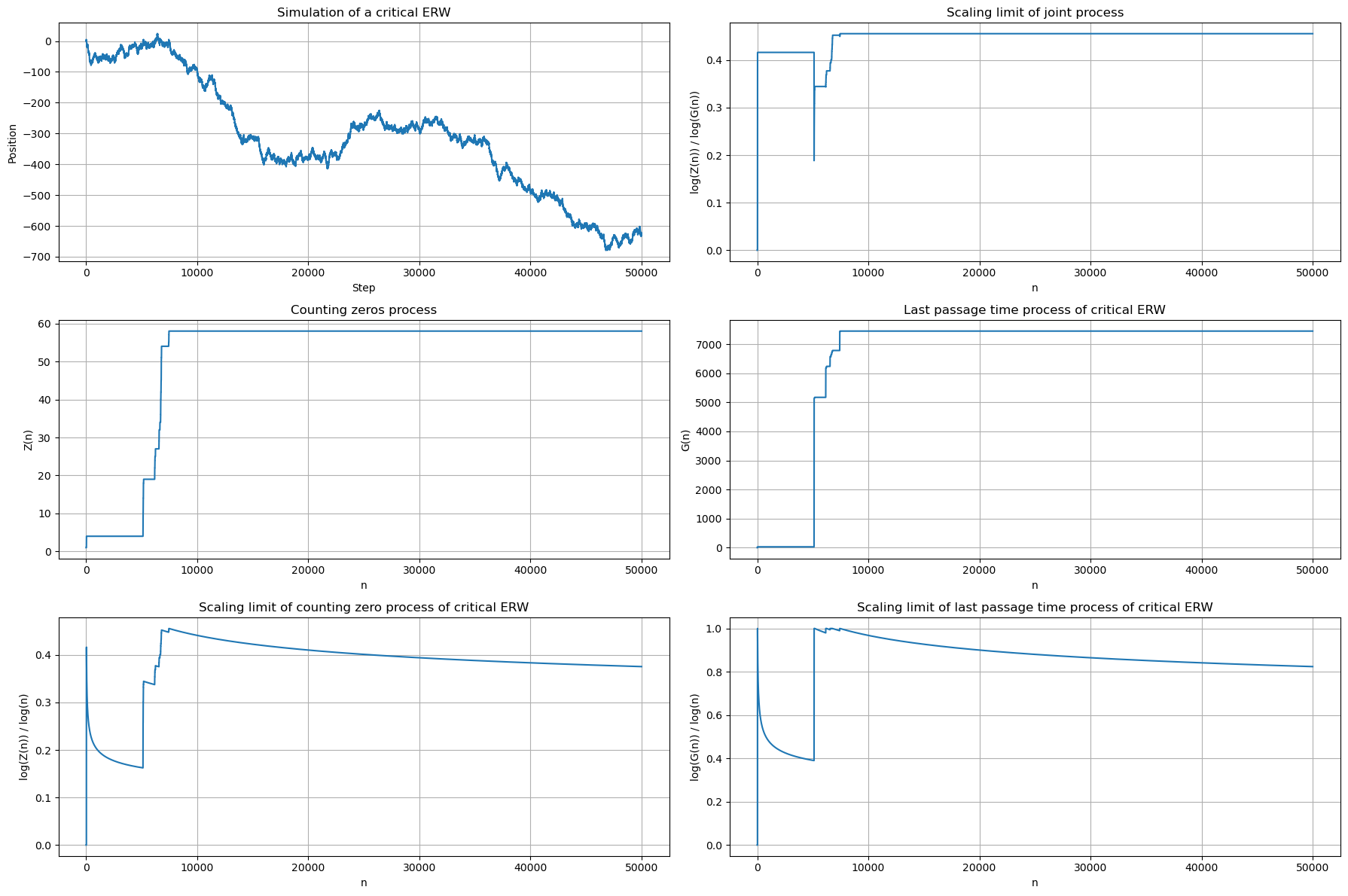}
  \caption{Simulation result for critical ERW with total step $n=50,000$ and its associating plots for rescaled counting zero process and rescaled last exit time process.}
  \label{Figure 1}
\end{figure}

\noindent Note $A$ has the same law as the last passage time at zero of standard Brownian motion before time $1$. The above theorem tells us under the logarithmic scale $Z(n)$ grows asymptotically at the speed of $n^{A/2}$, where $0<A\leq 1$ a.s. Meanwhile, the trend of convergence observed in the simulation titled "Scaling limit of counting zero process of critical ERW" in Figure \ref{Figure 1} aligns with our theoretical result. Moreover, (\ref{intro:convergence of Z(n)}) implies that, as $n$ approaches infinity, under logarithmic scale most zeros occur right before the last time the critical ERW leaves origin, which motivates us to define the process that describes the last time the critical ERW leaves the origin before time $n$, i.e.
\begin{equation*}
    G(n):= \max \{k\leq n: S(k)=0\}
\end{equation*}

\noindent As an important step that contributes to the proof of joint convergence in the later part of the paper, we briefly present the following result at this point. We obtain the limiting behaviour of the process $(G(n))_{n\geq 0}$,
\begin{equation*}\label{intro:convergence of G(n)}
    \lim_{n\rightarrow\infty} \log G(n)/ \log n = A \quad \text{in distribution}
\end{equation*}
The preceding result shows under logarithmic scale the last time the critical ERW leaves origin before step $n$ grows approximately like $n^A$, where $0<A\leq 1$ a.s. This is natural, as the last time the critical ERW leaves origin before $n$ is upper bounded by $n$. Then as the paramount of the work, we have joint convergence between the processes $G(n)$ and $Z(n)$. Here we will prove the following result

\begin{equation*}
       \lim_{n\rightarrow\infty}\frac{\log Z(n^t)}{\log G(n^t)}= \frac{1}{2} \quad \text{in probability}
    \end{equation*}

\noindent Observe that the simulation named "Scaling limit of joint process" in Figure \ref{Figure 1} supports our idea, i.e. the function has the trend of approaching $1/2$ as $n\rightarrow\infty$.

In this work, we also investigate the behaviour of the first return time of critical ERW, and here we denote it as
\begin{equation*}
    R:=\inf \{k\geq 1: S(k)=0\}
\end{equation*}

\begin{theorem}\label{intro: theorem 2}
The estimation for the tail distribution of the first return time is given as below 
\begin{equation}\label{intro: return time}
    \mathbb{P}(R>n) \sim \frac{2}{\pi}\sqrt{\frac{2}{ \log n}} \qquad \text{as $n \rightarrow\infty$}
\end{equation}
\end{theorem}
\noindent This result demonstrates a very heavy tail behaviour, indicating that the critical ERW returns to the origin slower than a 2-dimensional simple random walk. Specifically, for a 2-dimensional simple random walk, the tail distribution of the first return time decays asymptotically as $\pi/(2 \log n)$ as $n\rightarrow\infty$.

\begin{remark}
    Lots of work has also been carried out for the case of multidimensional elephant random walk (MERW), which is a natural extension of the $1$-dimensional ERW. It is a nearest neighbour process on $\mathbb{Z}^d$ for $d\geq 2$. The asymptotics of MERW has been addressed by \cite{{multidimensional ERW},{center of mass of the ERW},{functional limit}}. Recently the recurrence-transience property of MERW has been established in \cite{{Nicolas Curien}, {Shuo Qin}}. Unlike the $1$-dimensional ERW, MERW is transient at criticality, and is only recurrent in the subcritical regime when $d=2$. For $d\geq 3$, MERW is always transient for all memory parameters.
\end{remark}

The plan of the rest of paper is as follows: In Section \ref{section:preliminaries} we introduce the background for critical ERW, and as having been pointed out by \cite{invariance}, there exists a sequence of real numbers $(a_n)_{n\geq 0}$ such that $(a_nS(n))_{n\geq 0}$ is a martingale, and it can be embedded into a Brownian motion. In Section \ref{section:scaling limit}, we employ Brownian excursion theory to express the number of zeros visited by critical ERW in terms of the Brownian excursion intervals away from origin. And we establish a limiting result regarding the joint distribution between counting zero process and the process of the last time the critical ERW visits origin. In Section \ref{section:uniform estimates }, we conclude with the establishment of a concentration inequality for stopping times arising from the Brownian embedding and the derivation of a tail estimate for the first return time of the critical ERW.

Before ending the introduction, we point out a few conventions we adopt in notation throughout the text. In numerous instances, the bounds we provide for certain quantities rely on one or more variables, typically denoted as $k$ and $n$, while the constant component of these bounds, often denoted as $c$, remains independent of $k$ or $n$. Moreover, this constant may vary across different contexts. In some scenarios where the constant $c$ depends on further variables, such as $q$ (commonly encountered when applying the Burkholder-Davis-Gundy inequality), we denote it as $c_q$. Moreover, the function $o(1)$ used in various contexts should be considered as different functions that each approaches $0$ as the variable of each function tends to infinity. And we write $(f(t),g(t))\sim (h(t),m(t))$ as $t\rightarrow\infty$ if and only if we have both $\lim_{t\rightarrow\infty}f(t)/h(t)=1$ and $\lim_{t\rightarrow\infty}g(t)/m(t)=1$, where  $f, g, h, m$ are functions defined on the same space.

\section{Preliminaries for critical ERW and its Brownian embedding} \label{section:preliminaries}

  Before starting to prove our main result, we collect some facts and ideas relating specifically to critical ERW that we will use later.  

 Here, we assert that the distribution of the first step of critical ERW does not affect its visit to origin. This is because if we maintain the dynamics of the random walk $S$ while simply reversing the direction of the first step, we obtain its mirror-symmetric path $-S$, which does not alter the number of zeros visited by the random walk. Therefore, without loss of generality, we assume the first step of the random walk is Rademacher(1/2)-distributed.
We denote throughout the text $\mathbb{P}_0$ to be the law for the critical ERW.

Next, we define
\begin{equation*}
     a_{0}:=0 \quad \text{and} \quad a_{n}=\frac{\Gamma(n)}{\Gamma(n+1/2)} \text{  for $n\geq 1$}
 \end{equation*}
 Then by Stirling's formula, we have 
 \begin{equation}\label{a_n}
     a_{n}\sim n^{-1/2} \quad \text{as $n\rightarrow\infty$}
 \end{equation}

 The next lemma is a restatement of Lemma 2.1 in \cite{countingzeros}, which sums up basic properties for critical ERW we need for this work.

\begin{lemma}\label{lemma2.1}
For every even number $k\geq 0$, we have
 \begin{enumerate}[label={(\roman*)}]
 \item The process $(M({n}))_{n\geq 0}:=(a_{n}S(n))_{n\geq 0}$ is a martingale under $\mathbb{P}_0$.
 \item Its $(n+1)$-th increment
 \begin{equation*}
     \Delta M(n+1):= M(n+1)-M(n)
 \end{equation*}
 satisfies
 \begin{equation*}
     \Delta M(n+1)=-\frac{1}{2(n+1/2)}M(n)\pm a_{n+1}
 \end{equation*}
    \item For every $ q\geq 1$, one has
 \begin{equation*}
         \mathbb{E}(|M(n)|^{2q})\leq c_{q} (\log n)^{q}
    \end{equation*}
    \end{enumerate}
\end{lemma}

\begin{proof} The proof for (i) and (ii) are immediate from {\cite[Lemma 2.1]{countingzeros}} by setting $p=3/4$. Here we only prove (iii).
    Since, plainly $|M(n)|\leq n a_{n}$ and one has
    \begin{equation*}
        \left|\frac{1}{2(n+1/2)} M(n)\right|\leq a_{n+1}
    \end{equation*}
\noindent by (ii), we have $|\Delta M(n+1)|\leq 2a_{n+1}$, then we can bound the quadratic variation of $M$ from above by $\langle M,M \rangle(n)\leq 4A_{n}$, where $A_n$ is defined as
\begin{equation}\label{definition of An}
    A_{n}:=\sum_{i=1}^{n}a^2_{i}
\end{equation}
In light of (\ref{a_n}), we have $A_{n}\sim \log n$ as $n\rightarrow\infty$. We finish the proof with an application of Burkholder-Davis-Gundy inequality {\cite[see e.g. Chapter IV.4]{RY}}.
\end{proof}

Now we introduce the main tool for our work, namely Skorokhod embedding theorem for martingales, and the idea originates from \cite{invariance}. Here $(M(n))_{n\geq 0}$ is a martingale starting from the origin almost surely. We can construct a new probability space in which there exists a standard Brownian motion $(B(t))_{t\geq 0}$ and a sequence of increasing stopping times $(T_{n})_{n\geq 0}$ such that for every $n\geq 0$, $M(n)$ has the same distribution as $B(T_{n})$. By (ii) in Lemma \ref{lemma2.1}, we can define the aforementioned stopping times through a binary splitting martingale $(M(n))_{n\geq 0}$, as at every step it is supported by two values, moreover, we set $T_{0}:=0$, for every $n\geq 0$, we define recursively

 \begin{equation}\label{definition of T_{k,n}}
   T_{n+1}:=\inf \left\{t>T_{n}: B(t)-B(T_{n})=-\frac{1}{2(n+1/2)}B(T_{n})\pm a_{n+1}\right\}
\end{equation}

\noindent In this way, not only do we have
\begin{equation*}
    \left\{M(n), n\geq 0\right\}\overset{(d)}{=}\left\{B(T_{n}), n\geq 0\right\}
\end{equation*}

\noindent actually, without loss of generality, we may work with a realisation of ERW such that the above relation holds almost surely. 

The following lemma, which is a direct consequence of the above comment, is a cornerstone of our work that connects the zeros of the critical ERW and that of its Brownian embedding path. We first define the notation representing zero set for critical ERW:
\begin{equation*}
    \mathcal{Z}:=\{n\geq 0, S(n)=0\}
\end{equation*}

\begin{lemma}\label{locatezeros} {\cite[Lemma 2.2]{countingzeros}} The next inclusions hold $\mathbb{P}_0$-a.s
\begin{equation*}
    \{T_n, n\in\mathcal{Z}\}\subset \{t\geq 0, B(t)=0\}\subset \bigcup_{n\in\mathcal{Z}}[T_n, T_{n+1})
\end{equation*}
    
\end{lemma}

\begin{remark}\label{connection}
    In order to count actual zeros of the critical ERW from its Brownian embedding path, it requires us to find a way to filter out zeros from set $\{t\geq 0, B(t)=0\}\backslash \{T_n, n\in\mathcal{Z}\}$. This motivates us to introduce a selection criterion based on the absolute height of the excursion process of the Brownian embedding path in the next section.
\end{remark}

\section{Proof of Theorem \ref{intro: theorem 1}}\label{section:scaling limit}

The purpose of this section is to establish the joint convergence between the processes $G(n)$ and $Z(n)$ and the precise version of the convergence appeared in (\ref{intro:convergence of Z(n)}). We start this section by introducing a few notations that will be used throughout the upcoming text.\\
The counting zero process for the critical ERW before time $n$ can be expressed as
\begin{equation*}
    Z(n):=\text{Card}\{1\leq k\leq n: k\in\mathcal{Z}\}
\end{equation*}
Recall the definition of the last time the critical ERW leaves origin before time $n$, and we denote it as
\begin{equation}\label{G(n)}
    G^{S}(n):= \max \{k\leq n: S(k)=0\}
\end{equation}
Define the last exit time of the Brownian motion from origin before time $t$ as
\begin{equation}\label{g_t}
    G^B(t):=\sup\{s\leq t: B(s)=0\}
\end{equation}

Next we present our main result for this section, and the rest of the section is devoted to its proof.
\begin{proposition}\label{concentration of zero}
We have the following convergences.
\begin{enumerate}[label={\arabic*.}, ref=3.1.\arabic*]
\item \label{part 1} For every $t\geq 0$, 
    \begin{equation*}
        \frac{\log Z(n^t)}{\log G^{S}(n^t)}\xrightarrow[n\rightarrow\infty]{(\mathbb{P})} \frac{1}{2}
    \end{equation*}
\item \label{part 2}  The following weak convergence holds in the sense of finite dimensional distribution
    \begin{equation*}
         \left(\frac{\log G^{S}(n^t)}{\log n}\right)_{t\geq 0}\xrightarrow[n\rightarrow\infty]{(d)} \left(G^B(t)\right)_{t\geq 0}
    \end{equation*}
\item \label{part 3} The following weak convergence holds in the sense of finite dimensional distribution
    \begin{equation*}
       \left(\frac{\log(Z(n^t))}{\log n}\right)_{t\geq 0}\xrightarrow[n\rightarrow\infty]{(d)}\left(\frac{G^B(t)}{2}\right)_{t\geq 0}
    \end{equation*}
\end{enumerate}
\end{proposition}

Proposition \ref{part 3} is actually a generalisation of Theorem \ref{intro: theorem 1}. Thus, it suffices for us to prove the above proposition.

Plan of the section: we first prove Proposition \ref{part 2}, followed by a proof to Proposition \ref{part 1}, then we wrap up with a short proof to Proposition \ref{part 3}. 

The next lemma will be helpful in proving the Proposition \ref{part 2}.

\begin{lemma}\label{uniform convergence for last zero}
  Let \((f_n)_{n \geq 0}\) and \(f\) be real-valued functions on \(\mathbb{R}_+\). For all \(n \in\mathbb{N}\) and $s\in\mathbb{R}_+$, let \(g_n(s):=\sup\{r\leq s, f_n(r)=0\}\) and \(g(s):=\sup\{r\leq s, f(r)=0\}\). For fixed $t\in\mathbb{R}_+$ and suppose that the following conditions hold:
  \begin{enumerate}[label={(\roman*)}]
  \item The function \( f \) is continuous. And for any \( \delta > 0 \), there exist points $m,n\in (g(t)-\delta,g(t))$ such that $f(m)>0$ and $f(n)<0$.
  \item For each \(f_n\) and for any \(x, y \in \mathbb{R}_+\) such that \(f_n(x) > 0\) and \(f_n(y) < 0\), there exists \(z \in (x \wedge y, x \vee y)\) such that \(f_n(z) = 0\).
  \item \(f_n\) converges uniformly to \(f\) on \(\mathbb{R}_+\).
  \end{enumerate}
 
  \noindent then \(g_n(t)\) converges to \(g(t)\) as \(n \to \infty\).
\end{lemma}

    Condition (i) in Lemma \ref{uniform convergence for last zero} states that the function must take both strictly positive and strictly negative values immediately before the last time it visits the origin. For functions to fulfill (ii), even if they are not continuous, they must pass through the origin every time they change their sign. A concrete example would be the object we saw in (\ref{intro: critical regime convergence}), where $f$ is the standard 1-dimensional Brownian motion, and it obviously fulfills condition (i). We take the $f_n$'s to be the rescaled critical ERW, which indeed satisfies (ii). Recall also that the convergence to a continuous function in Skorokhod topology is equivalent to the uniform convergence on compact intervals.

\begin{proof}
For a fixed \( t \in \mathbb{R}_+ \) as mentioned above, we will first consider the case where $g(t)\neq t$. Without loss of generality, we assume \( f >  0 \) on \( (g(t), t] \). For every \( \delta > 0 \), by (ii), there exists \( m \in (g(t) - \delta, g(t)) \) such that \( f(m) < 0 \). Here, we take any \( l \in (g(t), g(t) + \delta) \), which by assumption we know \( f(l) > 0 \). We denote $\eta:= \min\left\{-f(m),  f(l), \inf_{s\in [l,t]}f(s)\right\}$. For every \( \epsilon \in \left(0, \eta\right) \), by (iii), there exists \( N \in \mathbb{N} \) such that for all \( n > N \), we have \( |f_{n}(m) - f(m)| < \epsilon \) and \( \sup_{s\in [l,t]}|f_{n}(s) - f(s)| < \epsilon \). Thus, for all \( n > N \), we have \( f_{n}(m) < 0 \) and \( f_{n}(s) > 0 \) for every $s\in [l,t]$. Combining with (i), we know \( g_n(t) \in (m,  l) \), which gives us \( g_n(t) \in (g(t) - \delta, g(t) + \delta) \). Note that since \(\delta > 0\) was chosen arbitrarily, this completes the proof. We then consider the case where $g(t)=t$, the conclusion follows directly from the above three conditions.
\end{proof}

\begin{proof}[Proof of Proposition \ref{part 2}]
\noindent Recall in (\ref{intro: critical regime convergence}), we have the following weak convergence in Skorokhod topology on $\mathbb{R}_+$ 
\begin{equation}\label{convergence in Skorokhod}
    \left(\frac{S\left({\lfloor{n^t}\rfloor}\right)}{\sqrt{\log n}n^{t/2}}, t\geq 0\right)\xrightarrow[n\rightarrow\infty]{(d)} (B(t), t\geq 0)
\end{equation}
where $(B_t, t\geq 0)$ is a standard 1-dimensional Brownian motion. Combining (\ref{convergence in Skorokhod}), Skorokhod representation theorem and the continuity of Brownian motion, we can choose a version of random walk $\left(S\left({\lfloor{n^t}\rfloor}\right); t\geq 0\right)$, which we denote as $\left(S^{(n)}\left({\lfloor{n^t}\rfloor}\right); t\geq 0\right)$ such that
\begin{equation}\label{version of S}
    \left(S^{(n)}\left({\lfloor{n^t}\rfloor}\right); t\geq 0\right)\overset{(d)}{=} \left(S\left({\lfloor{n^t}\rfloor}\right); t\geq 0\right) \quad\text{ for every $n\in\mathbb{N}$}
\end{equation}
and a version of Brownian motion $(B(t), t\geq 0)$, which we denote as $(\widetilde{B}(t), t\geq 0)$ such that
\begin{equation}\label{version of B}
    (\widetilde{B}(t), t\geq 0)\overset{(d)}{=}(B(t), t\geq 0)
\end{equation}
so that we have the following almost sure convergence, which holds with respect to the uniform topology

\begin{equation}\label{uniform cv}
    \left(\frac{S^{(n)}\left({\lfloor{n^t}\rfloor}\right)}{\sqrt{\log n}n^{t/2}}, t\geq 0\right)\xrightarrow{a.s.} (\widetilde{B}(t), t\geq 0)\quad\text{as $n\rightarrow\infty$}
\end{equation}
Thus, by (\ref{uniform cv}) and Lemma \ref{uniform convergence for last zero}, we obtain for every $t\in\mathbb{R}_+$ the last zero before time $t$ of the random walk $\left(S^{(n)}\left({\lfloor{n^t}\rfloor}\right); t\geq 0\right)$ converges almost surely to that of $(\widetilde{B}(t), t\geq 0)$ as $n\rightarrow\infty$. We conclude by putting together (\ref{version of S}), (\ref{version of B}), and the fact that for every $t\in\mathbb{R}_+$ and every $n$, 
$\frac{\log G^S(n^t)}{\log n}$ is the last time the rescaled critical ERW $\left(\frac{S\left({\lfloor{n^t}\rfloor}\right)}{\sqrt{\log n}n^{s/2}}, 0\leq s\leq t\right)$ visits $0$ before time $t$.
\end{proof}

To prepare ourselves with the proof of Proposition \ref{part 1}, we now introduce the selection method we mentioned in Remark \ref{connection}, as our task is to identify the zeros in the Brownian embedding path that originates from the critical ERW. Let's recall a compact interval $[l,r]$ is said to be an excursion interval for Brownian motion $B$ if and only if $B(l)=B(r)=0$ and $B(t)\neq 0$ for every $t\in (l,r)$. From here on we denote $l, r$ as the left and right extremity of the excursion intervals of the Brownian embedding path for the critical ERW. By Lemma \ref{locatezeros}, we deduce that for any excursion interval $[l,r]$, either its right-extremity is given by $r=T_n$ for some $n\in\mathcal{Z}$, then we say all such excursion intervals count, as the number of such intervals coincides with the number of times the critical ERW returns to the origin; or else the excursions will be contained in interval $(T_{n}, T_{n+1})$ for some $n\in\mathcal{Z}$, and we say all such $[l,r]$'s don't count, since these excursion intervals arise from the Brownian embedding path and not the ERW. Moreover, note $T_n$ here is a stopping time, by strong Markov property of Brownian motion, it will never be the left extremity of an excursion interval. Thus, this enables us to identify the number of zeros visited by critical ERW before time $n$ and the number of excursion intervals that count in the Brownian embedding trajectory on interval $[0,T_n]$. This leads us to define the next right-continuous process $(\alpha(t))_{t\geq 0}$,
\begin{equation*}
    \alpha(t):=a_{n+1}\quad \text{for $t\in [T_n, T_{n+1})$}
\end{equation*}

\noindent From the above construction, we know for every excursion interval $[l,r]\subset (T_n,T_{n+1})$, for some $n\in\mathcal{Z}$, we have, $\alpha(l)=a_{n+1}$ and $\max_{t\in[l,r]}|B(t)|< \alpha(l)$, since $T_{n+1}$ is the first time Brownian embedding path touches $\pm a_{n+1}$. That is to say, excursion intervals with $\max_{t\in[l,r]}|B(t)|\geq \alpha(l)$ are those not contained in $(T_n,T_{n+1})$ for some $n\in\mathcal{Z}$, i.e. they count.
We sum up the preceding lines with the next lemma.
\begin{lemma}{\cite[Lemma 3.3]{countingzeros}}\label{selection} For every $n\geq 0$, $Z(n)$ coincides with the number of excursion intervals $[l,r]\subset [0,T_{n}]$ with $\max_{t\in[l,r]}|B(t)|\geq \alpha(l)$.
\end{lemma}
We also need the approximations for $T_n$ and $\alpha(t)$ for the proof of Proposition \ref{part 1}.
\begin{lemma}\label{lemma3.4}
    The next two relations hold $\mathbb{P}_0$-a.s,
    \begin{enumerate}[label={(\roman*)}]
        \item $T_n = (1+o(1))\log n$ \quad \text{as $n\rightarrow\infty$}
        \item $\alpha(t)=\exp\left(-\frac{t}{2}(1+o(1))\right)$ \quad \text{as $t\rightarrow\infty$}
    \end{enumerate}
\end{lemma}

\begin{proof}
    (i) is a direct result from {\cite[Equation (12)]{invariance}}. And define the map  $T^{-1}:\mathbb{R}_+\rightarrow\mathbb{N}$ by $T^{-1}(t)=n+1$ for $T_n\leq t<T_{n+1}$. This is to help us locate which interval $t$ falls into, and with (i), we obtain
    \begin{equation*}
        T^{-1}(t)= \exp\left(t(1+o(1))\right)  \quad \text{as $t\rightarrow\infty$}
    \end{equation*}
    Then (ii) follows by the definition of $\alpha(t)$, i.e. $\alpha(t)=a_{T^{-1}(t)}\sim ({T^{-1}(t)})^{-1/2}$ as $t\rightarrow\infty$.
\end{proof}

Now in order to move forward, we provide some elements of Brownian excursion theory that will be of direct use in our proof. Please refer to Chapter XII in \cite{RY} for more detail. Here we denote $(L(t))_{t\geq 0}$ as the Brownian local time at level $0$, which is an increasing continuous adapted process that starts from origin at time $0$, and the support for its Stieljes measure $dL$ is the set $\{s\geq 0: B_{s}=0\}$. And for $u\geq 0$, we define the inverse local time at $u$ as $\tau_u:=\inf\{t\geq 0: L(t)>u\}$. It is a right-continuous process, so now we can identify all excursion intervals with $([\tau_{u-}, \tau_u])_{u\geq 0}$ whenever $\tau_u-\tau_{u-}>0$, and there are at most countably many such $u$'s. It is a well-known result that $\{(u,h_u), u\geq 0, \tau_u-\tau_{u-}>0\}$ is a Poisson point process with intensity $h^{-2}dhds$ on $\mathbb{R}_+\times \mathbb{R}$, where $ds$ and $dh$ are both Lebesgue measures and $h_u:=\max_{t\in [\tau_{u-}, \tau_u]}|B_t|$, which is the absolute height of the Brownian trajectory over the excursion interval. 

From the previous discussion, we define the following counting process
\begin{equation*}
    \nu(t):=\text{Card}\left(\{0<u\leq t: h_u\geq \alpha(\tau_{u-})\}\right) \quad \text{$ t\geq 0$}
\end{equation*}
The next lemma provides asymptotic behaviour of the above counting process.

\begin{lemma}\label{asyp for counting process}{\cite[Lemma 3.5]{countingzeros}}
    With probability one, we have
    \begin{equation*}
        \nu(t)\sim \int_{0}^{t}\frac{du}{\alpha(\tau_u)}\qquad\text{as $t\rightarrow\infty$}
    \end{equation*}
\end{lemma}
\noindent We next proceed with the proof of Proposition \ref{part 1}.
\begin{proof}[Proof of Proposition \ref{part 1}]
    In order to prove Proposition \ref{part 1}, it suffices for us to find, for every $t\geq 0$, the limiting law of 
    \begin{equation*}
        \left(\frac{\log G^S(n^t)}{\log n}, \frac{\log Z(n^t)}{\log n}\right)
    \end{equation*}
    We start by dealing with the second coordinate. Given Lemma \ref{selection}, we have $Z(n^t)=\nu(L(T_{\lfloor n^t \rfloor}))$. Together with Lemma \ref{asyp for counting process}, this yields
  \begin{align*}
        Z(n^t)\sim \int_{0}^{L(T_{\lfloor n^t \rfloor})}\frac{du}{\alpha(\tau_u)}=\int_{0}^{T_{\lfloor n^t \rfloor}}\frac{dL(u)}{\alpha(u)}
    \end{align*}
From a change of variable for Stieljes integral and Lemma \ref{lemma3.4} (i), we have
\begin{equation}\label{Z(n^t)}
    \int_{0}^{T_{\lfloor n^t \rfloor}}\frac{dL(u)}{\alpha(u)}\sim \int_{0}^{t}\frac{dL(u\log n)}{\alpha(u\log n)}
\end{equation}
Recall $G^B(t)$ defined in (\ref{g_t}) is the last time the standard Brownian motion leaves origin before time $t$. In light of Brownian scaling property, 
\begin{align}\label{local time scaling}
     \left(G^B(t \log n), L(t\log n)\right)_{t\geq 0}\overset{(d)}{=} \left(G^B(t )\log n, \sqrt{\log n}L(t)\right)_{t\geq 0}
\end{align}
and Lemma \ref{lemma3.4} (ii) gives us almost surely
\begin{equation}\label{asyalpha}
     \alpha(u\log n)= \exp\left(-\frac{u}{2}(1+o(1))\log n\right) \quad \text{as $n\rightarrow\infty$}
\end{equation}

\noindent Then for every $t\geq 0$, we can write
\begin{align}\label{Z(n^t)G(n^t)forjointconvergence}
    \left(\frac{\log G^S(n^t)}{\log n}, \frac{\log Z(n^t)}{\log n}\right) &\sim \left(\frac{G^B(t \log n)}{\log n}, \frac{1}{\log n}\log \int_{0}^{t}\frac{dL(u\log n)}{\alpha(u\log n)}\right) \text{  as $n\rightarrow\infty$}
\end{align}
where we used $\left\{M(n), n\geq 0\right\}\overset{a.s.}{=}\left\{B(T_{n}), n\geq 0\right\}$ and Lemma \ref{lemma3.4} (i) for the asymptotic equivalence in first coordinate while (\ref{Z(n^t)}) tells us that of the second coordinate. We observe, in (\ref{Z(n^t)G(n^t)forjointconvergence}), finding the limiting law of the left hand side is equivalent to finding that of the right hand side.

\noindent Because of (\ref{local time scaling}) and (\ref{asyalpha}), the right hand side of (\ref{Z(n^t)G(n^t)forjointconvergence}) has the same law as
\begin{equation}\label{important intermediate process}
    \left(G^B(t), \frac{1}{\log n}\left(\log\left(\sqrt{\log n}\right)+\log\int_{0}^{t}\exp\left(\left(\frac{u}{2}+o(1)\right)\log n\right)dL(u)\right)\right)  \text{  as $n\rightarrow\infty$}
\end{equation}
Next, we work on finding the almost sure limit of the second coordinate in (\ref{important intermediate process}). With an application of Brownian local time property, $L(t)=L(G^B(t))$, we have the following relations hold almost surely,
\begin{equation}\label{upper bound for joint convergence}
    \int_{0}^{t}\exp\left(\left(\frac{u}{2}+o(1)\right)\log n\right)dL(u)\leq \exp\left(\left(\frac{G^B(t)}{2}+o(1)\right)\log n\right)L(t)
\end{equation}
In the same fashion, let any $ \eta>0$
\begin{align}\label{lower bound for joint convergence}
    \int_{0}^{t}\exp & \left(\left(\frac{u}{2}+o(1)\right)\log n\right) dL(u)\geq \nonumber \\ &\exp \left(\left(\frac{G^B(t)-\eta}{2}+o(1)\right) \log n\right)(L(G^B(t))-L(G^B(t)-\eta))
\end{align}
since for every $\eta>0$, we have $L(G^B(t)-\eta)<L(G^B(t))$ almost surely. By putting together (\ref{important intermediate process}), (\ref{upper bound for joint convergence}), (\ref{lower bound for joint convergence}) and a use of the continuous mapping theorem, we finish the proof by first sending $n\rightarrow\infty$ then $\eta\rightarrow 0$.
\end{proof}

\begin{proof}[Proof of Proposition \ref{part 3}]
It follows from Proposition \ref{part 1}, \ref{part 2} together with Slutsky theorem.
\end{proof}

\section{Proof of Theorem \ref{intro: theorem 2}} \label{section:uniform estimates }

In this section, we begin by proving a concentration inequality for the distribution of stopping times that arose in the Brownian embedding of the critical ERW. Subsequently, we use this result to derive a tail estimate for the first return time of the critical ERW. 

Define the first time critical ERW comes back to origin, as
\begin{equation*}
    R:=\inf \{n\geq 1: S(n)=0\}
\end{equation*}

\begin{theorem}\label{theorem4.1}
    We have the following tail estimate for the first return time
    \begin{equation*}
        \lim_{n\rightarrow\infty}\sqrt{\log n}\mathbb{P}_0(R>n)={\frac{2\sqrt{2}}{\pi}}
    \end{equation*}
\end{theorem}

Notice in the next proposition the probability measure $\mathbb{P}$ represents the Wiener measure for the Brownian embedding path defined in Section \ref{section:preliminaries}. Recall $A_n:=\sum_{i=1}^{n}a^2_{i}$ which has been defined in (\ref{definition of An}) and $A_n\sim \log n$ as $n\rightarrow\infty$.

   \begin{proposition}\label{proposition4.2}
    For every $\epsilon>0$, $r\geq 1$, we have
   \begin{equation*}
        \mathbb{P}\left(\sup_{l\leq n}\left|T_{l}-A_{l}\right|\geq \epsilon \log n\right)\leq c_{\epsilon,r}\left(\frac{\log n}{n}\right)^r
    \end{equation*}
\end{proposition}

The first half of this section will be devoted to proving Proposition \ref{proposition4.2}.
Inspired by the definition of the stopping times in (\ref{definition of T_{k,n}}) that allows us to perform Brownian embedding, we define the following stopping time
\begin{equation*}
    \tau(x,y):=\inf \{t\geq 0: |B(t)+y|=x\}
\end{equation*}
which is the first time a Brownian motion starts from $y$ and leaves interval $(-x,x)$.

We write $(\mathcal{F}_{t})_{t\geq 0}$ the natural filtration induced by the Brownian motion $(B(t))_{t\geq 0}$. We define the increment between adjacent stopping times given in (\ref{definition of T_{k,n}}) as, for $n\geq 0$
\begin{equation*}
    \Delta T_{n+1}:=T_{n+1}-T_{n} = \inf \left\{t>0: B(T_{n}+t)-B(T_{n})=-\frac{1}{2(n+1/2)}B(T_{n})\pm a_{n+1}\right\} 
\end{equation*}

\noindent By conditioning $\Delta T_{n+1}$ on $\mathcal{F}_{T_{n}}$, the above line indicates on the event $\{B(T_{n})\}=b$, for some $b=a_{n}\mathbb{Z}$, the strong Markov property together with the definition of $\tau(x,y)$ gives us the conditional distribution of $\Delta T_{n+1}$ given $\mathcal{F}_{T_{n}}$ is the same as that of 
\begin{align*}
    \tau\left(a_{n+1},-\frac{1}{2(n+1/2)}b\right)
\end{align*}
Hence on the event $\{B(T_{n})\}=b$, for some $b=a_{n}\mathbb{Z}$, together with Lemma 4.4 (ii) in \cite{countingzeros} yields 
\begin{align*}
    \mathbb{E}\left(\Delta T_{n+1}|\mathcal{F}_{T_{n}} \right)= a^2_{n+1}-\frac{1}{4(n+1/2)^2} B^2(T_{n})
\end{align*}
It prompts us to introduce 
\begin{equation*}
    V(n):=\sum_{j=1}^{n}\frac{1}{4(n+1/2)^2} B^{2}(T_{j})
\end{equation*}
Thus, we can write
\begin{equation*}
    N(n):=T_{n}-A_{n}+V(n), \quad \text{for $n\geq 0$}
\end{equation*}
and it is a martingale.

\begin{lemma}\label{lemma4.5}
    For every integer $m\geq 1$, we have the following inequality
    \begin{equation*}
        \mathbb{E}\left(V(n)^{m}\right)\leq \frac{c_m }{n^{m/2}}
    \end{equation*}
\end{lemma}
\begin{proof} By the definition of $V_{k}(n)$, we can write
    \begin{align*}
        \mathbb{E}\left(V(n)^{m}\right)&=2^{-2m}\sum_{j_1,...j_m=1}^{n-1}\frac{\mathbb{E}\left(B^{2}(T_{j_1})\right)\cdots \mathbb{E}\left(B^{2}(T_{j_m})\right)}{\left(j_1+1/2\right)^2 \cdots\left(j_m+1/2\right)^2}\\ &\leq \sum_{j_1,...j_m=1}^{n-1}\frac{\left(\mathbb{E}\left(B^{2m}(T_{j_1})\right)\cdots \mathbb{E}\left(B^{2m}(T_{j_m})\right)\right)^{1/m}}{\left(j_1+1/2\right)^2 \cdots\left(j_m+1/2\right)^2}\\&\leq c_{m}\sum_{j_1,...j_m=1}^{n-1}\frac{\log j_1\cdots\log j_m}{\left(j_1+1/2\right)^2 \cdots\left(j_m+1/2\right)^2}\leq c_m \left(\sum_{j=1}^{n-1}\frac{1}{j^{3/2}}\right)^{m}
    \end{align*}
where in the second line we invoke H\"older inequality and we use Lemma \ref{lemma2.1} (iii) on the first inequality in the third line.
\end{proof}

The following upper bound from Lemma 4.6 in \cite{countingzeros} for $N_{k}(n)$ is crucial in our later proof, here for reader's convenience, we write it under the setting of $k=0$, $p=3/4$.

\begin{lemma}\label{lemma4.6}
    For every $q\geq 1$, we have
\begin{equation*}
        \mathbb{E}\left(\sup_{1\leq l\leq n}N(l)^{2q}\right)\leq c_q n^{-q}
    \end{equation*}
\end{lemma}

\noindent Now we are equipped to proving Proposition \ref{proposition4.2}.

\begin{proof}[Proof of Proposition \ref{proposition4.2}]
Since from the definition of $N(n)$, we have the following upper bound
\begin{equation*}
    \sup_{l\leq n}\left|T_{l}-A_{l}\right|\leq \sup_{l\leq n}\left|N(l)\right|+V(n)
\end{equation*}
    then followed by Lemma \ref{lemma4.5} and \ref{lemma4.6}, we know for $m\geq 1$,
 \begin{equation*}
        \mathbb{E}\left(\sup_{l\leq n}\left|T_{l}-A_{l}\right|^{2m}\right)\leq c_{m}n^{-m}
    \end{equation*}
    We finish the proof by an easy application of Markov inequality.
\end{proof}

\begin{proof}[Proof of Theorem \ref{theorem4.1}]
    This proof goes in the same spirit as in \cite{countingzeros} with minor differences. Here we still provide the proof for the sake of completeness. For every $s>0$, $n\geq 0$, we have
\begin{equation*}
        \{T_{R}\geq s\}=\{T_{R}\geq s, R\geq n\}\cup \{T_{R}\geq s, R<n\}\subset\{R\geq n\} \cup \{T_{n}\geq s\}
    \end{equation*}
    Hence, we have the lower bound
\begin{equation}\label{lowerboundforR}
        \mathbb{P}_0(R\geq n)\geq \mathbb{P}(T_{R}\geq s)-\mathbb{P}(T_{n}\geq s)
    \end{equation}
    Then by Lemma 4.7 in \cite{countingzeros}, we know, for any $s>0$
\begin{equation*}
        s  \mapsto  \left|\frac{\sqrt{\pi s}}{2}\mathbb{P}(T_{R}\geq s)-\sqrt{\frac{2}{\pi}}\right|
\end{equation*}
    converge to $0$, as $s\rightarrow\infty$. Thus, by taking $s=(1+\epsilon)A_n$, for any $\epsilon>0$, we obtain
    \begin{equation*}
        \lim_{n\rightarrow\infty}\frac{\sqrt{\pi \log n }}{{2}}\mathbb{P}\left(T_{R}\geq (1+\epsilon) A_n\right)=\sqrt{\frac{2}{\pi (1+\epsilon)}}
    \end{equation*}
    Also, by Proposition \ref{proposition4.2} together with taking $s=(1+\epsilon)A_{n}$, for any $\epsilon>0$, we have 
   \begin{equation*}
        \frac{\sqrt{\pi \log n}}{2}\mathbb{P}\left(T_{n}\geq (1+\epsilon)A_{n}\right)\leq \frac{\sqrt{\pi \log n}}{2}\mathbb{P}\left(T_{n}-A_{n}\geq \epsilon\log n\right)\xrightarrow{n\rightarrow\infty}0
    \end{equation*}
    where we use for every $n$, $A_n=\log n+\gamma+ O(\frac{1}{n})$, where $\gamma$ is the so called Euler-Macheroni constant and $O(\frac{1}{n})$ is a positive term, thus $A_n\geq \log n$. Together with (\ref{lowerboundforR}) gives us
    \begin{equation*}
        \liminf_{n\rightarrow\infty}\frac{\sqrt{\pi \log n}}{2}\mathbb{P}_0(R\geq n)\geq \sqrt{\frac{2}{\pi (1+\epsilon)}}
    \end{equation*}
    On the other hand, for every $s>0$, $n\geq 0$
    \begin{equation*}
        \{R \geq n\}=\{R \geq n, T_{n}\geq s\}\cup \{R \geq n, T_{n}<s\}\subset \{T_{R}\geq s\} \cup \{T_{n}<s\}
    \end{equation*}
     from which we have the upper bound,
    \begin{equation*}
        \mathbb{P}_0(R\geq n)\leq \mathbb{P}(T_{k,R}\geq s)+\mathbb{P}(T_{k,n} < s)
    \end{equation*}
    Then similar to the case for lower bound, we get
    \begin{equation*}
        \limsup_{n\rightarrow\infty}\frac{\sqrt{\pi \log n}}{2}\mathbb{P}_0(R\geq n)\leq \sqrt{\frac{2}{\pi (1-\epsilon)}}
   \end{equation*}
    Lastly, by letting $\epsilon\rightarrow 0$, we finish the proof.
\end{proof}

\section*{Acknowledgement} 
\noindent I would like to express my gratitude to my supervisor Jean Bertoin for his patience and all the valuable discussions I had with him during the time I wrote this article. Also, I would like to thank the anonymous referees for valuable comments that improved the quality of the paper.

\end{document}